\documentclass[11pt]{article}
\usepackage{amssymb,latexsym,amsmath,amsbsy,amsthm,amsxtra,amsgen,graphicx}
\newfont{\msbm}{msbm10 at 11pt}
\newcommand {\R} {\mbox{\msbm R}}

\newcommand {\N} {\mbox{\msbm N}}

\newtheorem{Theo}{Theorem}
\newtheorem{Lemma}[Theo]{Lemma}
\newtheorem{Cor}[Theo]{Corollary}
\newtheorem{Prop}[Theo]{Proposition}
\newtheorem{Exm}[Theo]{Example}

\begin{document}
\title{The number of small blocks in exchangeable random partitions}
\author{by Jason Schweinsberg\thanks{Supported in part by NSF Grant DMS-0805472}
\\ University of California, San Diego}
\maketitle



\vspace{-.5in}
\begin{abstract}
Suppose $\Pi$ is an exchangeable random partition of the positive integers and $\Pi_n$ is its restriction to $\{1, \dots, n\}$.  Let $K_n$ denote the number of blocks of $\Pi_n$, and let $K_{n,r}$ denote the number of blocks of $\Pi_n$ containing $r$ integers.  We show that if $0 < \alpha < 1$ and $K_n/(n^{\alpha} \ell(n))$ converges in probability to $\Gamma(1-\alpha)$, where $\ell$ is a slowly varying function, then $K_{n,r}/(n^{\alpha} \ell(n))$ converges in probability to $\alpha \Gamma(r - \alpha)/r!$.  This result was previously known when the convergence of $K_n/(n^{\alpha} \ell(n))$ holds almost surely, but the result under the hypothesis of convergence in probability has significant implications for coalescent theory.  We also show that a related conjecture for the case when $K_n$ grows only slightly slower than $n$ fails to be true.
\end{abstract}

\section{Introduction}
We begin by recalling some basic facts about exchangeable random partitions.  Suppose $\pi$ is a partition of the set $\N$ of positive integers.  If $\sigma$ is a permutation of $\N$, then we can define a partition $\sigma \pi$ such that the integers $\sigma(i)$ and $\sigma(j)$ are in the same block of $\sigma \pi$ if and only if $i$ and $j$ are in the same block of $\pi$.  A random partition $\Pi$ if $\N$ is said to be exchangeable if $\sigma \Pi$ and $\Pi$ have the same distribution for all permutations $\sigma$ of $\N$ having the property that $\sigma(j) = j$ for all but finitely many $j$.

In 1978, Kingman \cite{king78} proved an analog of de Finetti's Theorem that characterizes all possible exchangeable random partitions.  He showed that there is a one-to-one correspondence between distributions of exchangeable random partitions and probability measures on the infinite simplex $\Delta = \{(x_i)_{i=1}^{\infty}: x_1 \geq x_2 \geq \dots \geq 0, \sum_{i=1}^{\infty} x_i \leq 1\}$.  Given a probability distribution $\mu$ on $\Delta$, the associated exchangeable random partition is constructed as follows.  First, choose a random sequence $(P_j)_{j=1}^{\infty}$ with distribution $\mu$.  Then define random variables $(\xi_k)_{k=1}^{\infty}$ that are conditionally independent given $(P_j)_{j=1}^{\infty}$ and satisfy $P(\xi_k = i|(P_j)_{j=1}^{\infty}) = P_i$ and $P(\xi_k = -k|(P_j)_{j=1}^{\infty}) = 1 - \sum_{j=1}^{\infty} P_j$.  Finally, define $\Pi$ to be the random partition of $\N$ such that two integers $i$ and $j$ are in the same block of $\Pi$ if and only if $\xi_i = \xi_j$.

It follows from this construction and the Law of Large Numbers that if $B$ is a block of an exchangeable random partition $\Pi$, then the asymptotic frequency of the block, defined by $$\lim_{n \rightarrow \infty} \frac{1}{n} \sum_{i=1}^n {\bf 1}_{\{i \in B\}},$$ exists almost surely.  The nonzero asymptotic frequencies of the blocks of $\Pi$ are the nonzero terms of the sequence $(P_j)_{j=1}^{\infty}$.  Each integer is in a block having positive asymptotic frequency with probability $\sum_{j=1}^{\infty} P_j$ and is in a singleton block with probability $1 - \sum_{j=1}^{\infty} P_j$.

Given an exchangeable random partition $\Pi$ of $\N$, let $\Pi_n$ denote its restriction to $\{1, \dots, n\}$.  That is, $\Pi_n$ is the partition of $\{1, \dots, n\}$ such that two integers $i$ and $j$ in $\{1, \dots, n\}$ are in the same block of $\Pi_n$ if and only if they are in the same block of $\Pi$.  Let $K_n$ be the number of blocks of $\Pi_n$, and let $K_{n,r}$ be the number of blocks of $\Pi_n$ having size $r$.  In this paper, we show how asymptotic results for the random variables $K_{n,r}$ as $n \rightarrow \infty$ can be deduced from the asymptotic behavior of $K_n$.  Such results have already been proved, and are summarized in \cite{ghp}, for the case in which the asymptotic frequencies $P_j$ are deterministic and sum to one.  This is the setting of the classical infinite occupancy problem, in which infinitely many balls are placed independently into infinitely many boxes, with each ball going into the $j$th box with probability $P_j$.  Here we extend these results to the general case of random $P_j$ and explore the applications of this extension to coalescent theory and population genetics.

We note that in addition to the results below concerning the asymptotic behavior of $K_{n,r}$, Central Limit Theorems have been established for the number of small blocks in exchangeable random partitions under certain conditions.  See \cite{karlin} for some early work in this direction and \cite{bargne} for some recent extensions.

\subsection{The power law case}

We first consider the case in which the number of blocks $K_n$ grows like $n^{\alpha}$, where $0 < \alpha < 1$.  The proposition below is essentially due to Karlin \cite{karlin}.  More precisely, it follows from combining Theorem 1 of \cite{karlin} with a Tauberian theorem.  The result also appears as Corollary 21 in the recent survey \cite{ghp}.  Recall that a measurable function $\ell: (0, \infty) \rightarrow (0, \infty)$ is said to be slowly varying if for all $c > 0$, we have $\lim_{y \rightarrow \infty} \ell(cy)/\ell(y) = 1$.

\begin{Prop}\label{karlin1}
Let $(p_j)_{j=1}^{\infty}$ be a deterministic sequence such that $p_1 \geq p_2 \geq \dots \geq 0$ and $\sum_{j=1}^{\infty} p_j = 1$.  For $x > 0$, let $g(x) = \max\{j: p_j \geq x\}$.  Let $\Pi$ be an exchangeable random partition of $\N$ whose asymptotic block frequencies are given by $(p_j)_{j=1}^{\infty}$ almost surely, and define $K_n$ and $K_{n,r}$ as above.  Suppose $0 < \alpha < 1$.  Suppose $\ell: (0, \infty) \rightarrow (0, \infty)$ is a slowly varying function.  We have
\begin{equation}\label{k1a}
\lim_{x \rightarrow 0} \frac{x^{\alpha} g(x)}{\ell(1/x)} = 1
\end{equation}
if and only if
\begin{equation}\label{k1b}
\lim_{n \rightarrow \infty} \frac{K_n}{n^{\alpha} \ell(n)} = \Gamma(1 - \alpha) \hspace{.1in}\textup{a.s.}
\end{equation}
These two statements imply that for all $r \in \N$, we have
\begin{equation}\label{k1c}
\lim_{n \rightarrow \infty} \frac{K_{n,r}}{n^{\alpha} \ell(n)} = \frac{\alpha \Gamma(r - \alpha)}{r!} \hspace{.1in}\textup{a.s.}
\end{equation}
\end{Prop}

Our main theorem is an extension of Proposition \ref{karlin1} to general exchangeable random partitions.  It is an immediate consequence of Proposition \ref{karlin1} that even when the $P_j$ may be random, the condition (\ref{k1b}) implies (\ref{k1c}).  The result below says that this implication remains valid even when the convergence in (\ref{k1b}) holds only in probability.  As we will see shortly, this result has applications in coalescent theory, where it can be much easier to establish convergence in probability for $K_n$ than almost sure convergence.

\begin{Theo}\label{mainth}
Suppose $\Pi$ is an exchangeable random partition of $\N$, and define $K_n$ and $K_{n,r}$ as above.  Suppose $0 < \alpha < 1$, and suppose $\ell: (0, \infty) \rightarrow (0, \infty)$ is a slowly varying function.  If
\begin{equation}\label{s1b}
\lim_{n \rightarrow \infty} \frac{K_n}{n^{\alpha} \ell(n)} = \Gamma(1 - \alpha) \hspace{.1in}\textup{in probability}
\end{equation}
then for all $r \in \N$, we have
$$\lim_{n \rightarrow \infty} \frac{K_{n,r}}{n^{\alpha} \ell(n)} = \frac{\alpha \Gamma(r - \alpha)}{r!} \hspace{.1in}\textup{in probability}.$$
\end{Theo}

We prove Theorem \ref{mainth} in Section \ref{th1sec}.  It will follow from this proof (see Lemma \ref{GProb} below) that (\ref{s1b}) implies that the limit (\ref{k1a}) holds in probability.  However, the converse implication is false.  Of course, it is clear that the converse can not hold for general exchangeable random partitions because (\ref{k1a}) can hold even when $\sum_{j=1}^{\infty} p_j < 1$, in which case $K_n$ will be of order $n$ rather than of order $n^{\alpha}$.  However, as the next example shows, even under the additional condition that $\sum_{j=1}^{\infty} P_j = 1$, it is possible for the limit (\ref{k1a}) to hold in probability but for (\ref{s1b}) to fail.

\begin{Exm}\label{newex}
There exists an exchangeable random partition $\Pi$ of $\N$ whose asymptotic frequencies satisfy $\sum_{j=1}^{\infty} P_j = 1$ a.s. such that if $G(x) = \max\{j: P_j \geq x\}$, then $$\lim_{x \rightarrow 0} x^{\alpha} G(x) = 1 \hspace{.1in}\textup{in probability}$$ but $n^{-\alpha} K_n$ does not converge in probability to $\Gamma(1 - \alpha)$ as $n \rightarrow \infty$.
\end{Exm}

We describe the example in detail, and prove that it has the stated properties, in Section \ref{newexsec}.

\subsection{The case in which $K_n$ is only slightly smaller than $n$}

Proposition \ref{karlin1} and Theorem \ref{mainth} give asymptotic results for $K_{n,r}$ when $K_n$ grows like $n^{\alpha}$ for $0 < \alpha < 1$.  The result below concerns the case when $K_n$ grows just slightly slower than $n$.  This result can be obtained from results in \cite{ghp} by combining Propositions 14 and 18 with Lemma 1, Proposition 2, and the remarks before and after Proposition 2.

\begin{Prop}\label{karlin2}
Let $(p_j)_{j=1}^{\infty}$ be a deterministic sequence such that $p_1 \geq p_2 \geq \dots \geq 0$ and $\sum_{j=1}^{\infty} p_j = 1$.  For $x > 0$, let $g(x) = \max\{j: p_j \geq x\}$.  Let $\Pi$ be an exchangeable random partition of $\N$ whose asymptotic block frequencies are given by $(p_j)_{j=1}^{\infty}$ almost surely, and define $K_n$ and $K_{n,r}$ as above. 
Suppose $\ell: (0, \infty) \rightarrow (0, \infty)$ is a slowly varying function, and for $t > 0$, let $\ell_1(t) = \int_t^{\infty} \ell(s)/s \: ds$.  Suppose that
\begin{equation}\label{k2a}
\lim_{x \rightarrow 0} \frac{x g(x)}{\ell(1/x)} = 1.
\end{equation}
Then
\begin{equation}\label{k2b}
\lim_{n \rightarrow \infty} \frac{K_n}{n \ell_1(n)} = \lim_{n \rightarrow \infty} \frac{K_{n,1}}{n \ell_1(n)} = 1 \hspace{.1in}\textup{a.s.}
\end{equation}
Also, for integers $r \geq 2$,
\begin{equation}\label{k2c}
\lim_{n \rightarrow \infty} \frac{K_{n,r}}{n \ell(n)} = \frac{1}{r (r-1)} \hspace{.1in}\hspace{.1in}\textup{a.s.}
\end{equation}
\end{Prop}

Our next result addresses a question that is left open by Proposition \ref{karlin2}.  Although (\ref{k2a}) implies (\ref{k2b}) and (\ref{k2c}), one can also ask whether there is a result parallel to Theorem \ref{mainth} in which we obtain asymptotic results for $K_{n,r}$ just from the asymptotics of $K_n$.  However, the example below, which we describe in detail in Section \ref{th2sec}, shows that the condition $K_n/(n \ell_1(n)) \rightarrow 1$ a.s. is not sufficient to imply that the convergence in (\ref{k2c}) holds, even in probability.  Note that in the notation of Proposition \ref{karlin2}, if $\ell(t) = (\log t)^{-2}$ for all $t \geq T > 1$, then $\ell_1(t) = (\log t)^{-1}$ for all $t \geq T$.

\begin{Exm}\label{boszex}
There exists an exchangeable random partition $\Pi$ of $\N$ such that if $K_n$ and $K_{n,r}$ are defined as above, then $$\lim_{n \rightarrow \infty} \frac{(\log n) K_n}{n} = 1 \hspace{.1in}\textup{a.s.},$$ but for all integers $r \geq 2$, the quantity $n^{-1} (\log n)^2 K_{r,n}$ does not converge to $1/[r(r-1)]$ in probability as $n \rightarrow \infty$.
\end{Exm}

\subsection{Applications to coalescent theory and population genetics}

At first glance, Theorem \ref{mainth} may appear to be only a very minor technical improvement over Proposition \ref{karlin1}.  However, Theorem \ref{mainth} has significant implications for coalescent theory, where it can be much easier to prove convergence in probability and establish (\ref{s1b}) than to prove the almost sure convergence needed to obtain (\ref{k1b}).

Suppose we take a sample of size $n$ from a population and follow the ancestral lines of the sampled individuals backwards in time.  The ancestral lines will coalesce until all of the sampled individuals are traced back to a single common ancestor.  This process can be modeled by a stochastic process taking its values in the set of partitions of $\{1, \dots, n\}$.  The standard coalescent model is Kingman's coalescent \cite{king82}, in which it is assumed that only two lineages ever merge at a time and each transition that involves the merging of two lineages happens at rate one.  This means that when there are $b$ lineages, the amount of time before the next merger has an exponential distribution with rate $\binom{b}{2}$.  

Within the last decade, there has been considerable interest in alternative models of coalescence, called coalescents with multiple mergers or $\Lambda$-coalescents, that allow many ancestral lines to merge at once.  Such processes were introduced by Pitman \cite{pit99} and Sagitov \cite{sag99}.  If $\Lambda$ is a finite measure on $[0,1]$, then the $\Lambda$-coalescent is the coalescent process having the property that whenever there are $b$ lineages, each transition that involves $k$ lineages merging into one happens at rate $$\lambda_{b,k} = \int_0^1 x^{k-2} (1-x)^{b-k} \: \Lambda(dx).$$  Multiple mergers of ancestral lines could arise in populations with large family sizes, as many ancestral lines could be traced back to the individual that had a large number of offspring.  They could also arise as a result of natural selection because many ancestral lines could get traced back to an individual that had a beneficial mutation which spread rapidly to a large fraction of the population.

\begin{figure}
\hspace{-0.8in}\includegraphics{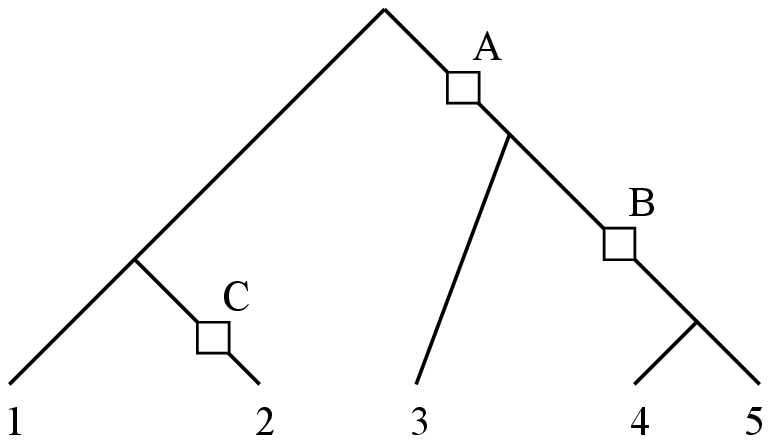}
\caption{}
\vspace{-9in}
{\small Figure 1: This figure shows the genealogy of five sampled individuals.  The boxes represent mutations.  Individual 1 inherited no mutations, individual 2 inherited mutation C, individual 3 inherited mutation A, and individuals 4 and 5 inherited mutations A and B.  Therefore, the allelic partition is $\Pi_5 = \{ \{1\}, \{2\}, \{3\}, \{4, 5\}\}$.  We have $K_5 = 4$.  Also, $K_{5,1} = 3$, $K_{5,2} = 1$, and $K_{5,3} = K_{5,4} = K_{5,5} = 0$.}
\end{figure}

To model mutations, we put marks representing mutations at points of a rate $\theta$ Poisson process along each branch of the coalescent tree.  One can then define a random partition $\Pi_n$ of $\{1, \dots, n\}$, often called the allelic partition, by declaring $i$ and $j$ to be in the same block of $\Pi_n$ if and only if the $i$th and $j$th sampled individuals inherit the same mutations.
These partitions $\Pi_n$ can be defined consistently as $n$ increases simply by sampling more individuals, so by Kolmogorov's Extension Theorem, on some probability space there is an exchangeable random partition $\Pi$ of $\N$ such that $\Pi_n$ is the restriction to $\Pi$ of $\{1, \dots, n\}$.  When the underlying coalescent process is Kingman's coalescent, the distribution of $\Pi_n$ is given by the Ewens Sampling Formula \cite{ewens}.  The probability that $\Pi_n$ has $a_j$ blocks of size $j$ for $j = 1, \dots, n$ is given by $$\frac{n!}{2 \theta (2 \theta + 1) \dots (2 \theta + n - 1)} \prod_{j=1}^n \bigg( \frac{2 \theta}{j} \bigg)^{a_j} \frac{1}{a_j!}.$$  When the underlying coalescent process is some other $\Lambda$-coalescent, there is no simple expression for the distribution of $\Pi$.  However, defining $K_n$ and $K_{n,r}$ from $\Pi_n$ as above, it was shown in \cite{bbs1, bbs2} that if $\Lambda$ is the Beta$(\alpha, 2 - \alpha)$ distribution with $0 < \alpha < 1$, then
\begin{equation}\label{beta1}
\lim_{n \rightarrow \infty} \frac{K_n}{n^{\alpha}} = \frac{\theta (2 - \alpha) (1 - \alpha) \Gamma(2 - \alpha)}{\alpha} \hspace{.1in}\textup{in probability}.
\end{equation}
It was then shown in \cite{bbs2} that
\begin{equation}\label{beta2}
\lim_{n \rightarrow \infty} \frac{K_{n,r}}{n^{\alpha}} = \frac{\theta (2 - \alpha) (1 - \alpha)^2 \Gamma(r - \alpha)}{r!} \hspace{.1in}\textup{in probability}.
\end{equation}
Note that $\alpha$ here corresponds to $2 - \alpha$ in \cite{bbs1} and \cite{bbs2}.  The proof of (\ref{beta2}) in \cite{bbs2} is rather technical, exploiting a connection between beta coalescents and the genealogy of continuous-state branching processes.  However, Theorem \ref{mainth} makes it possible to deduce (\ref{beta2}) immediately from (\ref{beta1}).  We also note that the convergence in (\ref{beta1}) was later shown in \cite{bbl} to hold almost surely, allowing (\ref{beta2}) to be established via Proposition \ref{karlin1}.  On the other hand, if $\Lambda$ is the uniform distribution on $[0, 1]$, corresponding to $\alpha = 1$ above, it was shown in \cite{bago07}, building on work of \cite{dimr07}, that
\begin{equation}\label{bosz1}
\lim_{n \rightarrow \infty} \frac{(\log n) K_n}{n} = \theta \hspace{.1in}\textup{in probability}.
\end{equation}
It was also shown in \cite{bago07} that 
\begin{equation}\label{bosz2}
\lim_{n \rightarrow \infty} \frac{(\log n)^2 K_{n,r}}{n} = \frac{\theta}{k(k-1)} \hspace{.1in}\textup{in probability}.
\end{equation}
However, Example \ref{boszex} establishes that (\ref{bosz1}) does not imply (\ref{bosz2}).  Indeed, the proof of (\ref{bosz2}) in \cite{bago07} involves a detailed analysis of a Markov chain on different time scales.

\subsection{A model of a growing population}

To illustrate another application of Theorem \ref{mainth}, we consider the following model of a population that grows in size over time.  Fix $\gamma > 0$ and a positive integer $N$.  Assume that for each positive integer $k$, there are $\lceil N k^{-\gamma} \rceil$ individuals in generation $-k$.  For simplicity, assume that the number of individuals in generation zero is the same as the number of individuals in generation -1, so there are $N$ individuals in generations $0$ and $1$ but fewer in earlier generations.  To give the model a genealogical structure, we assume, as in the standard Wright-Fisher model, that each individual chooses its parent uniformly at random from the individuals in the previous generation.

Now sample $n$ individuals from the population at time zero, and follow their ancestral lines backwards in time.  We can represent the genealogy of these sampled individuals by a coalescent process $(\Psi_{N,n}(t), t \geq 0)$ taking its values in the set of partitions of $\{1, \dots, n\}$, where two integers $i$ and $j$ are in the same block of the partition $\Psi_{N,n}(t)$ if and only if the $i$th and $j$th individuals in the sample have the same ancestor at time $-\lfloor N^{1/(1 + \gamma)} t \rfloor$.  It is easy to check that as $N \rightarrow \infty$, these processes converge to a coalescent process $(\Psi_n(t), t \geq 0)$ having the property that at time $t$, two lineages (that is, two blocks of the partition) are merging at rate $t^{\gamma}$.  To see this, note that in generation $N^{1/(1 + \gamma)} t$, two individuals have the same ancestor with probability approximately $N^{-1} (N^{1/(1 + \gamma)} t)^{\gamma}$, and multiplying this expression by the time-scaling factor $N^{1/(1 + \gamma)}$ gives the coalescence rate of $t^{\gamma}$.  Note that $(\Psi_n(t), t \geq 0)$ is a time-inhomogeneous Markov chain.

The process $(\Psi_n(t), t \geq 0)$ can be obtained as a time-change of Kingman's coalescent.
Indeed, let $(\Theta_n(t), t \geq 0)$ be Kingman's coalescent started with $n$ lineages.  That is $(\Theta_n(t), t \geq 0)$ is a continuous-time, time-homogeneous Markov chain taking values in the set of partitions of $\{1, \dots, n\}$ such that $\Theta_n(0) = \{\{1\}, \{2\}, \dots, \{n\}\}$, each transition that involves merging two blocks of the partition happens at rate one, and no other transitions are possible.  Then we can define
\begin{equation}\label{timechange}
\Psi_n(t) = \Theta_n \bigg( \frac{t^{\gamma + 1}}{\gamma + 1} \bigg).
\end{equation}
The time change makes $\Psi_n$ a time-inhomogeneous Markov chain in which at time $t$, each pair of blocks is merging at rate $t^{\gamma}$.

We will now work with the coalescent process $(\Psi_n(t), t \geq 0)$ and, as before, put mutations along each lineage at times of a rate $\theta$ Poisson process.  Then define the partition $\Pi_n$ such that $i$ and $j$ are in the same block of $\Pi_n$ if and only if the $i$th and $j$th sampled individuals inherit the same mutations.  The partitions $\Pi_n$ can be defined consistently as $n$ varies, so there is an exchangeable random partition $\Pi$ of $\N$ such that $\Pi_n$ is the restriction of $\Pi$ to $\{1, \dots, n\}$.  Define $K_n$ and $K_{n,r}$ as before.  We obtain the following result.

\begin{Theo}\label{coalth}
Consider the time-inhomogeneous coalescent process with mutations described above.  Let $\alpha = \gamma/(1 + \gamma) \in (0, 1)$.  We have
\begin{equation}\label{growpop}
\lim_{n \rightarrow \infty} \frac{K_n}{n^{\alpha}} = \frac{\theta 2^{1-\alpha} (1-\alpha)^{\alpha} \pi}{\sin(\pi \alpha)} \hspace{.1in}\textup{in probability}
\end{equation}
and for all $r \in \N$,
\begin{equation}\label{growpop2}
\lim_{n \rightarrow \infty} \frac{K_{n,r}}{n^{\alpha}} = \frac{\theta 2^{1-\alpha} (1-\alpha)^{\alpha} \pi}{\sin(\pi \alpha)} \cdot \frac{\alpha \Gamma(r - \alpha)}{r! \Gamma(1 - \alpha)} \hspace{.1in}\textup{in probability}.
\end{equation}
\end{Theo}

Of course, in view of Theorem \ref{mainth}, equation (\ref{growpop2}) follows immediately from (\ref{growpop}), so we need only prove (\ref{growpop}), which we do in Section \ref{appsec}.

Note that for both the beta coalescent and for the time-inhomogeneous coalescent described above, we have
\begin{equation}\label{afs}
\lim_{n \rightarrow \infty} \frac{K_{n,r}}{K_n} = \frac{\alpha \Gamma(r - \alpha)}{r! \Gamma(1 - \alpha)} \hspace{.1in}\textup{in probability}.
\end{equation}
The left-hand side of (\ref{afs}) is the fraction of blocks of the allelic partition having size $r$, and the sequence of numbers $K_{n,r}$ for $1 \leq r \leq n$ is often called the allele frequency spectrum.  Thus, (\ref{afs}) says that we get the same allele frequency spectrum for these two models, as we would with any coalescent model having the property that $K_n$ grows like $n^{\alpha}$.

One of the central goals of population genetics is to use information about a sample from a current population to obtain information about the history of the population.  Distinguishing among various factors that could cause the genealogy of the population to differ from Kingman's coalescent can be challenging.  See, for example, \cite{jensen} and \cite{stephan} for a discussion of the issue of distinguishing the effects of natural selection from demographic factors such as changing population size.  Therefore, from the perspective of population genetics, Theorem \ref{coalth} is perhaps disappointing.  Theorem \ref{coalth} shows that the allele frequency spectrum that arises when the genealogy is given by a beta coalescent, as could be the case for populations with large family sizes, could also arise in a population whose size is increasing over time.  Thus, one can not necessarily use the allele frequency spectrum to distinguish populations with large family sizes from populations that are increasing in size.  In general, Proposition \ref{karlin1} and Theorem \ref{mainth} suggest that the same allele frequency spectrum may arise in a wide variety of models, and thus may explain part of the difficulty in distinguishing among various factors that could cause the genealogy of a population to differ from Kingman's coalescent.

\section{Proof of Theorem \ref{mainth}}\label{th1sec}

Throughout this section, we assume that $0 < \alpha < 1$ and that $\Pi$ is an exchangeable random partition.  We define $K_n$ and $K_{n,r}$ as in Theorem \ref{mainth}.  We assume that $\ell: (0, \infty) \rightarrow (0, \infty)$ is a slowly varying function and that (\ref{s1b}) holds.  We denote by $P_1 \geq P_2 \geq \dots$ the asymptotic frequencies of the blocks of $\Pi$.  Note that (\ref{s1b}) implies that $\sum_{j=1}^{\infty} P_j = 1$ a.s. because $\liminf_{n \rightarrow \infty} n^{-1} K_n > 0$ almost surely on the event that $\sum_{j=1}^{\infty} P_j < 1$.  For $x > 0$, define $G(x) = \max\{j: P_j \geq x\}$, which is a random variable because the $P_j$ are random.

At times in the proof of Theorem \ref{mainth}, it will be useful to use a technique called Poissonization.  Let $(N(t), t \geq 0)$ be a rate one Poisson process, so that $N(t)$ has the Poisson distribution with mean $t$ for all $t$.  Define the random variable $$\Phi(t) = E[K_{N(t)}|(P_j)_{j=1}^{\infty}].$$  Likewise, for positive integers $r$, define $$\Phi_r(t) = E[K_{N(t), r}|(P_j)_{j=1}^{\infty}].$$  We have (see the proof of Proposition 17 in \cite{ghp}),
\begin{equation}\label{A}
\Phi(t) = t \int_0^{\infty} e^{-tx} G(x) \: dx \hspace{.1in}\textup{a.s.}
\end{equation}
Also,
\begin{equation}\label{F}
\Phi_r(t) = \frac{t^r}{r!} \sum_{j=1}^{\infty} P_j^r e^{-t P_j} \hspace{.1in}\textup{a.s.}
\end{equation}
By conditioning on $(P_j)_{j=1}^{\infty}$ and applying Lemma 1 and Proposition 2 of \cite{ghp}, we get
\begin{equation}\label{B}
\lim_{n \rightarrow \infty} \frac{K_n}{\Phi(n)} = 1 \hspace{.1in}\textup{a.s.}
\end{equation}
Using the remarks following Proposition 2 of \cite{ghp}, we have for all positive integers $r$,
\begin{equation}\label{C}
\lim_{n \rightarrow \infty} \frac{\sum_{s=r}^{\infty} K_{n,s}}{\sum_{s=r}^{\infty} \Phi_s(n)} = 1 \hspace{.1in}\textup{a.s.}
\end{equation}

Lemma \ref{slowvar} below, known as Potter's Theorem, is Theorem 1.5.6(i) of \cite{regvar} and gives some bounds on slowly varying functions.  Note that since Theorem \ref{mainth} only concerns the values of $\ell(n)$ for $n \in \N$, we may and will assume, here and throughout this section, that $\ell$ is bounded away from zero and infinity on $(0, x]$ for any $x > 0$.

\begin{Lemma}\label{slowvar}
Suppose $\ell: (0, \infty) \rightarrow (0, \infty)$ is a slowly varying function.  Let $\delta > 0$.  There exists a positive number $x_0(\delta)$ such that if $x \geq x_0(\delta)$ and $\lambda \geq 1$, then
\begin{equation}\label{x0eq}
\frac{1}{(1 + \delta) \lambda^{\delta}} \leq \frac{\ell(\lambda x)}{\ell(x)} \leq (1 + \delta) \lambda^{\delta}.
\end{equation}
Also, there exists a constant $C > 0$ such that $\ell(x) \geq C x^{-\delta}$ for all $x \geq x_0(\delta)$.
\end{Lemma}

\begin{Lemma}\label{Glem}
We have $$\lim_{t \rightarrow \infty} \frac{t^{1 - \alpha}}{\ell(t)} \int_0^{\infty} e^{-tx} G(x) \: dx = \Gamma(1 - \alpha) \hspace{.1in}\textup{in probability}.$$
\end{Lemma}

\begin{proof}
We use Poissonization.  Combining (\ref{s1b}) and (\ref{B}), we get
$$\lim_{n \rightarrow \infty} \frac{\Phi(n)}{n^{\alpha} \ell(n)} = \Gamma(1 - \alpha) \hspace{.1in}\textup{in probability}.$$  Since $t \mapsto \Phi(t)$ is nondecreasing and $\ell$ is slowly varying, it follows from Lemma \ref{slowvar} that $$\lim_{t \rightarrow \infty} \frac{\Phi(t)}{t^{\alpha} \ell(t)} = \Gamma(1 - \alpha) \hspace{.1in}\textup{in probability}.$$  The result now follows from (\ref{A}).
\end{proof}

\begin{Lemma}\label{zeroint}
We have $$\lim_{t \rightarrow \infty} \frac{t^{1 - \alpha}}{\ell(t)} \int_0^{\infty} e^{-tx} (G(x) - x^{-\alpha} \ell(1/x)) \: dx = 0 \hspace{.1in}\textup{in probability}.$$
\end{Lemma}

\begin{proof}
In view of Lemma \ref{Glem}, it suffices to show that
\begin{equation}\label{gammaint}
\lim_{t \rightarrow \infty} \frac{t^{1 - \alpha}}{\ell(t)} \int_0^{\infty} e^{-tx} x^{-\alpha} \ell(1/x) \: dx = \Gamma(1 - \alpha).
\end{equation}
Choose $\delta$ such that $\delta + \alpha < 1$, and choose $x_0$ so that (\ref{x0eq}) holds for $x \geq x_0$ and $\lambda \geq 1$.  Substituting $y = tx$, we get
\begin{align}\label{zi1}
\frac{t^{1 - \alpha}}{\ell(t)} \int_0^{\infty} e^{-tx} &x^{-\alpha} \ell(1/x) \: dx = \frac{1}{\ell(t)} \int_0^{\infty} e^{-y} y^{-\alpha} \ell(t/y) \: dy \nonumber \\
&= \int_0^{\infty} e^{-y} y^{-\alpha} \bigg( \frac{\ell(t/y)}{\ell(t)} \bigg) {\bf 1}_{\{y \leq t/x_0\}} \: dy + \frac{1}{\ell(t)} \int_{t/x_0}^{\infty} e^{-y} y^{-\alpha} \ell(t/y) \: dy.
\end{align}
By (\ref{x0eq}), we have $\ell(t/y)/\ell(t) \leq \max\{2 y^{-\delta}, 2 y^{\delta}\}$ whenever $t \geq x_0$ and $0 < y \leq t/x_0$.  Also, since $\ell$ is a slowly varying function, $\lim_{t \rightarrow \infty} \ell(t/y)/\ell(t) = 1$ for all $y > 0$.  Therefore, by the Dominated Convergence Theorem,
\begin{equation}\label{zi2}
\lim_{t \rightarrow \infty} \int_0^{\infty} e^{-y} y^{-\alpha} \bigg( \frac{\ell(t/y)}{\ell(t)} \bigg) {\bf 1}_{\{y \leq t/x_0\}} \: dy = \int_0^{\infty} e^{-y} y^{-\alpha} \: dy = \Gamma(1-\alpha).
\end{equation}
Recall from Lemma \ref{slowvar} that there is a constant $C$ such that $\ell(t) \geq C t^{-\delta}$ for all $t \geq x_0$.  By assumption, there is a constant $B$ such that $\ell(x) \leq B$ for $0 < x \leq x_0$.  Therefore,
\begin{equation}\label{zi3}
\limsup_{t \rightarrow \infty} \frac{1}{\ell(t)} \int_{t/x_0}^{\infty} e^{-y} y^{-\alpha} \ell(t/y) \: dy \leq \limsup_{t \rightarrow \infty} \frac{B t^{\delta}}{C} \bigg( \frac{t}{x_0} \bigg)^{-\alpha} e^{-t/x_0} = 0.
\end{equation}
Equation (\ref{gammaint}) follows from (\ref{zi1}), (\ref{zi2}), and (\ref{zi3}).
\end{proof}

\begin{Lemma}\label{bigGlem}
There exists a positive number $C_0$ such that if $C > C_0$, then $$\lim_{x \rightarrow 0} P(G(x) > C \ell(1/x) x^{-\alpha}) = 0.$$
\end{Lemma}

\begin{proof}
Suppose $G(x) \geq C \ell(1/x) x^{-\alpha}$.  Since $x \mapsto G(x)$ is nonincreasing,
$$t^{1-\alpha} \int_0^{\infty} e^{-ty} G(y) \: dy \geq t^{1-\alpha} \int_0^x e^{-ty} C \ell(1/x) x^{-\alpha} \: dy = C (tx)^{-\alpha} \ell(1/x) (1 - e^{-tx}).$$
Therefore, if $t = 1/x$, then $$\frac{t^{1-\alpha}}{\ell(t)} \int_0^{\infty} e^{-ty} G(y) \: dy \geq C(1 - e^{-1}).$$  By Lemma \ref{Glem}, the result follows with $C_0 = \Gamma(1 - \alpha)/(1 - e^{-1})$.
\end{proof}

\begin{Lemma}\label{Gbarlem}
There exists a positive number $C$ such that if ${\bar G}(x) = G(x) {\bf 1}_{\{G(x) > C \ell(1/x) x^{-\alpha}\}}$, then $$\lim_{ t \rightarrow \infty} \frac{t^{1 - \alpha}}{\ell(t)} \int_0^{\infty} e^{-tx} {\bar G}(x) \: dx = 0 \hspace{.1in}\textup{in probability}.$$
\end{Lemma}

\begin{proof}
Let $\epsilon > 0$.  Choose $C_1 > C_0$, where $C_0$ is the constant from Lemma \ref{bigGlem}, and let $C = 2^{1 + \alpha} C_1$.  Choose an integer $M$ large enough that $C_1 2^{-M \alpha} e^{-2^M} < \epsilon/2$ and $2^{-M (1 - \alpha)} \Gamma(1 - \alpha) e < \epsilon/4$.  For $t > 0$, define the event $$A_t = \{G(2^k t^{-1})\leq C_1 \ell(2^{-k} t) (2^k t^{-1})^{-\alpha} \mbox{ for }k = -M, -M + 1, \dots, M-1, M\}.$$  By Lemma \ref{bigGlem}, there exists $T_1 < \infty$ such that if $t > T_1$, then $P(A_t) > 1 - \epsilon/2$.  Because ${\bar G}(x) \leq G(x) \leq G(2^M t^{-1})$ for all $x \geq 2^M t^{-1}$, on the event $A_t$ we have
\begin{align}\label{pre1}
\frac{t^{1-\alpha}}{\ell(t)} \int_{2^M t^{-1}}^{\infty} e^{-tx} {\bar G}(x) \: dx &\leq \frac{t^{1-\alpha}}{\ell(t)} \cdot C_1 \ell(2^{-M} t)(2^M t^{-1})^{-\alpha} \int_{2^M t^{-1}}^{\infty} e^{-tx} \: dx \nonumber \\
&= C_1 2^{-M \alpha} e^{-2^M} \cdot \frac{\ell(2^{-M} t)}{\ell(t)} < \frac{\ell(2^{-M} t)}{\ell(t)} \cdot \frac{\epsilon}{2}.
\end{align}
Because $\ell$ is slowly varying, $\ell(2^{-M}t)/\ell(t) \rightarrow 1$ as $t \rightarrow \infty$.  Therefore, there exists a $T_2$ such that for $t > T_2$, on $A_t$ we have
\begin{equation}\label{piece1}
\frac{t^{1-\alpha}}{\ell(t)} \int_{2^M t^{-1}}^{\infty} e^{-tx} {\bar G}(x) \: dx < \frac{\epsilon}{2}.
\end{equation}

Also, on $A_t$, if $2^k t^{-1} \leq x \leq 2^{k+1} t^{-1}$ for some integer $k$ satisfying $-M \leq k \leq M-1$, then $$G(x) \leq G(2^k t^{-1}) \leq C_1 \ell(2^{-k} t)(2^k t^{-1})^{-\alpha} \leq C_1 \ell(2^{-k} t)(x/2)^{-\alpha} = 2^{\alpha} C_1 \ell(2^{-k} t) x^{-\alpha}.$$  By Lemma \ref{slowvar}, there exists $T_3 < \infty$ such that if $t > T_3$ and $2^k t^{-1} \leq x \leq 2^{k+1} t^{-1}$ for some integer $k$ satisfying $-M \leq k \leq M-1$, then $\ell(2^{-k} t)/\ell(1/x) \leq 2$.  Therefore, if $A_t$ occurs and $t > T_3$ then $$G(x) \leq 2^{1 + \alpha} C_1 \ell(1/x) x^{-\alpha} = C \ell(1/x) x^{-\alpha}.$$  In this case, ${\bar G}(x) = 0$ for $2^{-M}t^{-1} \leq x \leq 2^M t^{-1}$ and thus
\begin{equation}\label{piece2}
\frac{t^{1-\alpha}}{\ell(t)} \int_{2^{-M} t^{-1}}^{2^M t^{-1}} e^{-tx} {\bar G}(x) \: dx = 0.
\end{equation}

If $0 \leq x \leq 2^{-M} t^{-1}$, then $e^{-tx} \leq 1 \leq e \cdot e^{-2^M tx}$.  Therefore,
\begin{align}\label{pre3}
\frac{t^{1-\alpha}}{\ell(t)} \int_0^{2^{-M} t^{-1}} e^{-tx} {\bar G}(x) \: dx &\leq \frac{e t^{1-\alpha}}{\ell(t)} \int_0^{2^{-M} t^{-1}} e^{-2^M tx} {\bar G}(x) \: dx \nonumber \\
&\leq \bigg( e 2^{-M(1-\alpha)} \cdot \frac{\ell(2^M t)}{\ell(t)} \bigg) \frac{(2^M t)^{1-\alpha}}{\ell(2^M t)} \int_0^{\infty} e^{-2^M tx} G(x) \: dx \nonumber \\
&\leq \bigg( \frac{\epsilon}{4 \Gamma(1-\alpha)} \cdot \frac{\ell(2^M t)}{\ell(t)} \bigg) \frac{(2^M t)^{1-\alpha}}{\ell(2^M t)} \int_0^{\infty} e^{-2^M tx} G(x) \: dx.
\end{align}
By Lemma \ref{Glem} with $2^M t$ in place of $t$, the portion of the right-hand side of (\ref{pre3}) after the parentheses converges in probability to $\Gamma(1-\alpha)$ as $t \rightarrow \infty$.  Also, because $\ell$ is slowly varying, we have $\ell(2^M t)/\ell(t) \rightarrow 1$ as $t \rightarrow \infty$.  Therefore, there exists $T_4$ such that if $t > T_4$, then
\begin{equation}\label{piece3}
P \bigg( \frac{t^{1-\alpha}}{\ell(t)} \int_0^{2^{-M} t^{-1}} e^{-tx} {\bar G}(x) \: dx > \frac{\epsilon}{2} \bigg) < \frac{\epsilon}{2}.
\end{equation}
It follows from (\ref{piece1}), (\ref{piece2}), and (\ref{piece3}) that if $t > \max\{T_1, T_2, T_3, T_4\}$, then $$P \bigg( \frac{t^{1-\alpha}}{\ell(t)} \int_0^{\infty} e^{-tx} {\bar G}(x) \: dx > \epsilon \bigg) < \epsilon,$$ which implies the lemma.
\end{proof}

\begin{Lemma}\label{GProb}
We have $$\lim_{x \rightarrow 0} \frac{x^{\alpha} G(x)}{\ell(1/x)} = 1 \hspace{.1in}\textup{in probability}.$$
\end{Lemma}

\begin{proof}
Choose $C_0$ as in Lemma \ref{bigGlem}, and choose $C > \max\{C_0, 1\}$ large enough that the conclusion of Lemma \ref{Gbarlem} holds.  For $x \geq 0$, let $$Y(x) = \min\{G(x), C \ell(1/x)x^{-\alpha}\} - x^{-\alpha} \ell(1/x).$$  In view of Lemma \ref{bigGlem}, it suffices to show that
\begin{equation}\label{Ynts}
\lim_{x \rightarrow 0} \frac{x^{\alpha} Y(x)}{\ell(1/x)} = 0 \hspace{.1in}\textup{in probability}.
\end{equation}
Note that $|Y(x)| \leq C x^{-\alpha} \ell(1/x)$ for all $x \geq 0$.  By Lemmas \ref{zeroint} and \ref{Gbarlem},
\begin{equation}\label{Ycp}
\lim_{t \rightarrow \infty} \frac{t^{1-\alpha}}{\ell(t)} \int_0^{\infty} e^{-tx} Y(x) \: dx = 0 \hspace{.1in}\textup{in probability}.
\end{equation}

We proceed by contradiction.  Suppose (\ref{Ynts}) fails to hold.  Then there exists $0 < \epsilon < 1/2$ and a sequence of positive numbers $(s_n)_{n=1}^{\infty}$ converging to zero such that one of the following holds:
\begin{enumerate}
\item We have $P(Y(s_n) > \epsilon s_n^{-\alpha} \ell(1/s_n)) > \epsilon$ for all $n$.

\item We have $P(Y(s_n) < - \epsilon s_n^{-\alpha} \ell(1/s_n)) > \epsilon$ for all $n$.
\end{enumerate}

Assume for now that we are in the first case, so $P(Y(s_n) > \epsilon s_n^{-\alpha} \ell(1/s_n)) > \epsilon$ for all $n$.  If $Y(s_n) > \epsilon s_n^{-\alpha} \ell(1/s_n)$, then $G(s_n) > (1 + \epsilon)s_n^{-\alpha} \ell(1/s_n)$.  In this case, if $x < s_n$, we have $$G(x) \geq G(s_n) > (1 + \epsilon) s_n^{-\alpha} \ell(1/s_n) = (1 + \epsilon) \bigg( \frac{x}{s_n} \bigg)^{\alpha} \frac{\ell(1/s_n)}{\ell(1/x)} \cdot x^{-\alpha} \ell(1/x).$$  Choose $\delta > 0$ small enough that $(1 + \delta)^{-1}(1 + \epsilon)^{1 - \alpha - \delta} > 1$.  If $s_n/(1 + \epsilon) < x < s_n$ and if $n$ is large enough that $1/s_n > x_0(\delta)$, then by Lemma \ref{slowvar}, $$\frac{1}{(1 + \delta)(1 + \epsilon)^{\delta}} \leq \frac{\ell(1/s_n)}{\ell(1/x)} \leq (1 + \delta)(1 + \epsilon)^{\delta}.$$  Therefore, $$G(x) > \frac{(1 + \epsilon)^{1 - \alpha - \delta}}{(1 + \delta)} x^{-\alpha} \ell(1/x).$$  It follows that for $s_n/(1 + \epsilon) < x < s_n$, we have
\begin{align}\label{Yeta}
Y(x) &> \bigg( \frac{(1 + \epsilon)^{1 - \alpha - \delta}}{(1 + \delta)} - 1 \bigg) x^{-\alpha} \ell(1/x) \nonumber \\
&\geq  \bigg( \frac{(1 + \epsilon)^{1 - \alpha - \delta}}{(1 + \delta)} - 1 \bigg) \frac{1}{(1 + \delta)(1 + \epsilon)^{\delta}} s_n^{-\alpha} \ell(1/s_n) \nonumber \\
&= \eta s_n^{-\alpha} \ell(1/s_n),
\end{align}
where $\eta > 0$.

Let $f: [0, \infty) \rightarrow \R$ be the function such that $f(x) = 0$ if either $x \leq 1/(1+\epsilon)$ or $x \geq 1$, $f((2 + \epsilon)/(2 + 2 \epsilon)) = 1$, and $f$ is linear on the two intervals $[1/(1 + \epsilon), (2+\epsilon)/(2 + 2 \epsilon)]$ and $[(2+\epsilon)/(2 + 2 \epsilon), 1]$.  Note that 
\begin{equation}\label{triangle}
\int_{1/(1 + \epsilon)}^1 f(x) \: dx = \frac{1}{2} \bigg(1 - \frac{1}{1 + \epsilon} \bigg) = \frac{\epsilon}{2(1 + \epsilon)}.
\end{equation}
Let ${\cal A}$ be the algebra of functions of the form $\varphi(x) = a_1 e^{-t_1 x} + \dots + a_m e^{-t_m x}$ for $x \geq 0$, where $m$ is a nonnegative integer, $a_1, \dots, a_m \in \R$, and $t_1, \dots, t_m \geq 1$.  By the Stone-Weierstrass Theorem (see, for example, Theorem D.23 on p. 346 of \cite{cohn}), the set ${\cal A}$ is uniformly dense in the set $C_0([0, \infty))$ of continuous functions from $[0, \infty)$ to $\R$ that vanish at infinity.  Therefore, if we choose $\zeta = \epsilon \eta/(16 \Gamma(1 - \alpha) C)$, then there is a function $g \in {\cal A}$ such that $|g(x) - e^x f(x)| \leq \zeta$ for all $x \geq 0$.  Letting $h(x) = e^{-x} g(x)$ for $x \geq 0$, we have $|h(x) - f(x)| \leq \zeta e^{-x}$ for all $x \geq 0$.  Write $g(x) = a_1 e^{-t_1 x} + \dots + a_m e^{-t_m x}$.  

Choose $\theta = \min\{\epsilon/2m, 2^{1 - \alpha} \epsilon \eta/8(|a_1| + \dots + |a_m|)\} > 0$.  By (\ref{Ycp}) we can choose $n$ large enough that $2/s_n \geq T$, where for $t \geq T$ we have $$P \bigg( \bigg| \frac{t^{1-\alpha}}{\ell(t)} \int_0^{\infty} e^{-tx} Y(x) \: dx \bigg| > \theta \bigg) < \theta.$$  It follows that with probability at least $1 - m \theta$,
\begin{align}\label{hYbound}
\bigg| \int_0^{\infty} h(x/s_n) Y(x) \: dx \bigg| &= \bigg| \sum_{i=1}^m a_i \int_0^{\infty} e^{-(t_i + 1)x/s_n} Y(x) \: dx \bigg| \nonumber \\
&\leq \sum_{i=1}^m |a_i| \bigg( \frac{t_i + 1}{s_n} \bigg)^{\alpha - 1} \ell \bigg( \frac{t_i + 1}{s_n} \bigg) \theta \nonumber \\
&\leq \frac{\theta}{2^{1-\alpha}} s_n^{1-\alpha} \ell(1/s_n) \sum_{i=1}^m |a_i| \frac{\ell((t_i+1)/s_n)}{\ell(1/s_n)}.
\end{align}
Also, using (\ref{gammaint}) with $1/s_n$ in place of $t$, we have that for sufficiently large $n$,
\begin{align}\label{fhdiff}
\bigg| \int_0^{\infty}  \big(f(x/s_n) - h(x/s_n) \big) Y(x) \: dx \bigg| &\leq \int_0^{\infty} \zeta e^{-x/s_n} C x^{-\alpha} \ell(1/x) \: dx \nonumber \\
&\leq 2 C \Gamma(1 - \alpha) \zeta s_n^{1-\alpha} \ell(1/s_n).
\end{align}
Since $\ell$ is slowly varying, it follows from (\ref{hYbound}) and (\ref{fhdiff}) that with probability at least $1 - m \theta$,
\begin{equation}\label{limsupeq}
\limsup_{n \rightarrow \infty} \frac{1}{s_n^{1 - \alpha} \ell(1/s_n)} \bigg| \int_0^{\infty} f(x/s_n) Y(x) \: dx \bigg| \leq 2 \Gamma(1 - \alpha) C \zeta + \frac{\theta (|a_1| + \dots + |a_m|)}{2^{1-\alpha}} \leq \frac{\epsilon \eta}{4}.
\end{equation}
However, (\ref{Yeta}) and (\ref{triangle}) imply that for sufficiently large $n$, with probability at least $\epsilon$,
\begin{align}
\int_0^{\infty} f(x/s_n) Y(x) \: dx &= \int_{s_n/(1 + \epsilon)}^{s_n} f(x/s_n) Y(x) \: dx \nonumber \\
&> \eta s_n^{-\alpha} \ell(1/s_n) \int_{s_n/(1 + \epsilon)}^{s_n} f(x/s_n) \: dx \nonumber \\
&= \frac{\epsilon \eta}{2 (1 + \epsilon)} s_n^{1-\alpha} \ell(1/s_n), \nonumber
\end{align}
which contradicts (\ref{limsupeq}) because $m \theta < \epsilon$.

It remains now to consider the second case.  Assume that $P(Y(s_n) < - \epsilon s_n^{-\alpha} \ell(1/s_n)) > \epsilon$ for all $n$.  If $Y(s_n) < - \epsilon s_n^{-\alpha} \ell(1/s_n)$, then $G(s_n) < (1 - \epsilon) s_n^{-\alpha} \ell(1/s_n)$.  In this case, if $x > s_n$, then $$G(x) \leq G(s_n) < (1 - \epsilon) s_n^{-\alpha} \ell(1/s_n) = (1 - \epsilon) \bigg( \frac{x}{s_n} \bigg)^{\alpha} \frac{\ell(1/s_n)}{\ell(1/x)} \cdot x^{-\alpha} \ell(1/x).$$  Choose $\delta > 0$ small enough that $(1 + \delta)(1 - \epsilon)^{1 - \alpha - \delta} < 1$.  If $s_n < x < s_n/(1-\epsilon)$ and if $n$ is large enough that $(1 - \epsilon)/s_n > x_0(\delta)$, then by Lemma \ref{slowvar}, $$\frac{(1 - \epsilon)^{\delta}}{1 + \delta} \leq \frac{\ell(1/s_n)}{\ell(1/x)} \leq \frac{1 + \delta}{(1 - \epsilon)^{\delta}}.$$  Therefore, $$G(x) < (1 + \delta) (1 - \epsilon)^{1 - \alpha - \delta} x^{-\alpha} \ell(1/x).$$  It follows that for $s_n < x < s_n/(1-\epsilon)$, we have
\begin{align}\label{Yupper}
Y(x) &< \big( (1 + \delta)(1 - \epsilon)^{1 - \alpha - \delta} - 1 \big) x^{-\alpha} \ell(1/x) \nonumber \\
&\leq \big( (1 + \delta)(1 - \epsilon)^{1 - \alpha - \delta} - 1 \big) (1 + \delta)^{-1}(1 - \epsilon)^{\alpha + \delta} s_n^{-\alpha} \ell(1/s_n) \nonumber \\
&= -\eta s_n^{-\alpha} \ell(1/s_n),
\end{align}
where $\eta > 0$.

This time, let $f: [0, \infty) \rightarrow \R$ be the function such that $f(x) = 0$ if $x \leq 1$ or $x \geq 1/(1-\epsilon)$, $f((2-\epsilon)/(2-2\epsilon)) = 1$, and $f$ is linear on $[1, (2-\epsilon)/(2-2\epsilon)]$ and $[(2-\epsilon)/(2-2\epsilon), 1/(1-\epsilon)]$.  We have
\begin{equation}\label{fint}
\int_1^{1/(1-\epsilon)} f(x) \: dx = \frac{1}{2} \bigg(\frac{1}{1 - \epsilon} - 1 \bigg) = \frac{\epsilon}{2(1-\epsilon)}.
\end{equation}
Define $g$ and $\theta$ as in the previous case.  Then (\ref{hYbound}), (\ref{fhdiff}), and (\ref{limsupeq}) hold as before.  However, (\ref{Yupper}) and (\ref{fint}) imply that for sufficiently large $n$, with probability at least $\epsilon$,
\begin{align}
\int_0^{\infty} f(x/s_n) Y(x) \: dx &= \int_{s_n}^{s_n/(1-\epsilon)} f(x/s_n) Y(x) \: dx \nonumber \\
&< - \eta s_n^{-\alpha} \ell(1/s_n) \int_{s_n}^{s_n/(1-\epsilon)} f(x/s_n) \: dx \nonumber \\
&= -\frac{\epsilon \eta}{2(1-\epsilon)} s_n^{1-\alpha} \ell(s_n)
\end{align}
which again contradicts (\ref{limsupeq}) because $m \theta < \epsilon$.
\end{proof}

\begin{proof}[Proof of Theorem \ref{mainth}]
Fix $r \in \N$.  It follows from (\ref{s1b}) and (\ref{C}) that given $\epsilon > 0$, for sufficiently large $n$ we have
\begin{equation}\label{D}
P \bigg( \bigg| \sum_{s=r}^{\infty} K_{n,s} - \sum_{s=r}^{\infty} \Psi_s(n) \bigg| < \frac{\epsilon}{2} n^{\alpha} \ell(n) \bigg) > 1 - \frac{\epsilon}{2}
\end{equation}
and
\begin{equation}\label{E}
P \bigg( \bigg| \sum_{s=r+1}^{\infty} K_{n,s} - \sum_{s=r+1}^{\infty} \Psi_s(n) \bigg| < \frac{\epsilon}{2} n^{\alpha} \ell(n) \bigg) > 1 - \frac{\epsilon}{2}
\end{equation}
Subtracting (\ref{E}) from (\ref{D}) gives that
$$P \bigg( \bigg| \frac{K_{n,r}}{n^{\alpha} \ell(n)} - \frac{\Phi_r(n)}{n^{\alpha} \ell(n)} \bigg| < \epsilon \bigg) > 1 - \epsilon$$ for sufficiently large $n$,  Therefore, it suffices to show that
\begin{equation}\label{phirprob}
\lim_{t \rightarrow \infty} \frac{\Phi_r(t)}{t^{\alpha} \ell(t)} = \frac{\alpha \Gamma(r - \alpha)}{r!} \hspace{.1in}\textup{in probability}.
\end{equation}

Let $\theta > 0$ be arbitrary.  Because $\sum_{r=1}^{\infty} \alpha \Gamma(r - \alpha)/r! = \Gamma(1-\alpha)$, we can choose $N$ large enough that
\begin{equation}\label{Nchoose}
\sum_{r=1}^N \frac{\alpha \Gamma(r - \alpha)}{r!} > \Gamma(1 - \alpha) - \frac{\theta}{2}.
\end{equation}
Let $\eta = \min\{\theta/(N+1), \theta/(4 \Gamma(1 - \alpha))\}$.  Note that we can choose a sufficiently large integer $L$, then a sufficiently small positive number $\delta$ (much smaller than $1/L$), then a sufficiently large integer $M$ (much larger than $1/\delta$), then a sufficiently small positive number $\epsilon$ (much smaller than $1/M$) such that
\begin{equation}\label{4C}
\bigg( \frac{L}{L+2} \bigg)^r \bigg((1 - \epsilon)^2 - \frac{4 \epsilon M}{\alpha} \bigg) \int_{\delta (L+1)}^{\delta (M+1)} e^{-y} y^{r-\alpha-1} \: dy > (1 - \eta) \Gamma(r - \alpha)
\end{equation}
for $1 \leq r \leq N$.  By Lemma \ref{GProb}, we can choose $T_1 > 0$ sufficiently large that if $t \geq T_1$, then $$P\big( (1 - \epsilon) x^{-\alpha} \ell(1/x) \leq G(x) \leq (1 + \epsilon) x^{-\alpha} \ell(1/x) \mbox{ for }x = L\delta/t, (L+1)\delta/t, \dots, M\delta/t \big) > 1 - \eta.$$  By Lemma \ref{slowvar}, we can choose $T_2 > 0$ sufficiently large that if $t \geq T_2$ and $L \delta/t \leq x \leq M \delta/t$, then $$1 - \epsilon \leq \frac{\ell(1/x)}{\ell(t)} \leq 1 + \epsilon.$$

If $t \geq \max\{T_1, T_2\}$, then with probability at least $1 - \eta$, we have, using (\ref{F}),
\begin{align}\label{lower}
\frac{\Phi_r(t)}{t^{\alpha} \ell(t)} &= \frac{t^{r - \alpha}}{r! \ell(t)} \sum_{j=1}^{\infty} P_j^r e^{-t P_j} \nonumber \\
&\geq \frac{t^{r - \alpha}}{r! \ell(t)} \sum_{k = L}^{M-1} \bigg( \frac{k \delta}{t} \bigg)^r e^{-(k+1) \delta} \big( G(k \delta/t) - G((k+1)\delta/t) \big) \nonumber \\
&\geq \frac{t^{-\alpha} \delta^r}{r! \ell(t)} \sum_{k = L}^{M-1} k^r e^{-(k+1) \delta} \bigg( (1 - \epsilon) \bigg( \frac{k \delta}{t} \bigg)^{-\alpha} \ell \bigg( \frac{t}{k \delta} \bigg) - (1 + \epsilon) \bigg( \frac{(k+1) \delta}{t} \bigg)^{-\alpha} \ell \bigg( \frac{t}{(k+1)\delta} \bigg)\bigg) \nonumber \\
&\geq \frac{\delta^{r-\alpha}}{r!} \sum_{k=L}^{M-1} k^r e^{-(k+1) \delta} \bigg( \frac{(1 - \epsilon)^2}{k^{\alpha}} - \frac{(1 + \epsilon)^2}{(k + 1)^{\alpha}} \bigg).
\end{align}
For $k \leq M-1$,
\begin{align}
\frac{(1 - \epsilon)^2}{k^{\alpha}} - \frac{(1 + \epsilon)^2}{(k +1)^{\alpha}} &= (1 - \epsilon)^2 \bigg( \frac{1}{k^{\alpha}} - \frac{1}{(k+1)^{\alpha}} \bigg) + \frac{1}{(k+1)^{\alpha}} \big( (1 - \epsilon)^2 - (1 + \epsilon)^2 \big) \nonumber \\
&\geq \frac{\alpha (1 - \epsilon)^2}{(k+1)^{\alpha + 1}} - \frac{4 \epsilon}{(k+1)^{\alpha}} \nonumber \\
&\geq \frac{\alpha}{(k+1)^{\alpha + 1}} \bigg( (1 - \epsilon)^2 - \frac{4 \epsilon M}{\alpha} \bigg).\nonumber
\end{align}
Therefore, if $1 \leq r \leq N$ and $t \geq \max\{T_1, T_2\}$, then with probability at least $1 - \eta$,
$$\frac{\Phi_r(t)}{t^{\alpha} \ell(t)} \geq \bigg( (1 - \epsilon)^2 - \frac{4 \epsilon M}{\alpha} \bigg) \frac{\alpha \delta^{r-\alpha}}{r!} \sum_{k=L}^{M-1} \frac{k^r}{(k+1)^{\alpha+1}} e^{-(k+1) \delta}.$$
If $r \geq 2$ and $k \geq L$ then $$\frac{k^r}{(k+1)^{\alpha + 1}} e^{-(k+1) \delta} \geq \bigg( \frac{L}{L+2} \bigg)^r (k+2)^{r - \alpha - 1} e^{-(k+1) \delta} \geq \bigg( \frac{L}{L+2} \bigg)^r \int_{k+1}^{k+2} x^{r-\alpha-1} e^{-\delta x} \: dx,$$ and if $r = 1$ and $k  \geq L$ then
$$\frac{k^r}{(k+1)^{\alpha + 1}} e^{-(k+1) \delta} \geq \bigg( \frac{L}{L+1} \bigg)^r (k+1)^{r-\alpha-1} e^{-(k+1) \delta} \geq \bigg( \frac{L}{L+2} \bigg)^r \int_{k+1}^{k+2} x^{r-\alpha-1} e^{-\delta x} \: dx.$$
Thus, if $1 \leq r \leq N$ and $t \geq \max\{T_1, T_2\}$, then with probability at least $1 - \eta$, we have
\begin{align}\label{G}
\frac{\Phi_r(t)}{t^{\alpha} \ell(t)} &\geq \bigg( \frac{L}{L+2} \bigg)^r  \bigg( (1 - \epsilon)^2 - \frac{4 \epsilon M}{\alpha} \bigg) \frac{\alpha \delta^{r-\alpha}}{r!} \int_{L+1}^{M+1} x^{r-\alpha-1} e^{-\delta x} \: dx \nonumber \\
&= \bigg( \frac{L}{L+2} \bigg)^r  \bigg( (1 - \epsilon)^2 - \frac{4 \epsilon M}{\alpha} \bigg) \frac{\alpha}{r!} \int_{\delta(L+1)}^{\delta(M+1)} e^{-y} y^{r-\alpha-1} \: dy. \nonumber \\
&> \frac{(1 - \eta) \alpha \Gamma(r - \alpha)}{r!},
\end{align}
where the last inequality uses (\ref{4C}).

Since $t \mapsto \Phi(t)$ is nondecreasing and $\ell$ is slowly varying, (\ref{s1b}) and (\ref{B}) imply that $\Phi(t)/(t^{\alpha} \ell(t))$ converges in probability to $\Gamma(1 - \alpha)$ as $t \rightarrow \infty$.  Therefore, there exists $T_3$ such that if $t \geq T_3$, then
\begin{equation}\label{H}
P \bigg( \frac{\Phi(t)}{t^{\alpha} \ell(t)} \leq (1 + \eta) \Gamma(1 - \alpha) \bigg) > 1 - \eta.
\end{equation}
Therefore, combining (\ref{G}) and (\ref{H}), if $1 \leq r \leq N$ and $t \geq \max\{T_1, T_2, T_3\}$, then with probability at least $1 - (N+1)\eta \geq 1 - \theta$,
\begin{align}\label{upper}
\frac{\Phi_r(t)}{t^{\alpha} \ell(t)} &\leq \frac{1}{t^{\alpha} \ell(t)} \bigg( \Phi(t) - \sum_{\substack{s = 1 \\ s \neq r}}^{N} \Phi_s(t) \bigg) \nonumber \\
&\leq (1 + \eta) \Gamma(1-\alpha) - (1 - \eta) \sum_{\substack{s = 1 \\ s \neq r}}^{N} \frac{\alpha \Gamma(s - \alpha)}{s!} \nonumber \\
&= \Gamma(1 - \alpha) - \sum_{\substack{s = 1 \\ s \neq r}}^{N} \frac{\alpha \Gamma(s - \alpha)}{s!} + \eta \bigg( \Gamma(1 - \alpha) + \sum_{\substack{s = 1 \\ s \neq r}}^{N} \frac{\alpha \Gamma(s - \alpha)}{s!} \bigg) \nonumber \\
&\leq \frac{\alpha \Gamma(r - \alpha)}{r!} + \frac{\theta}{2} + 2 \eta \Gamma(1 - \alpha) \nonumber \\
&\leq \frac{\alpha \Gamma(r - \alpha)}{r!} + \theta.
\end{align}
using (\ref{Nchoose}).  The result (\ref{phirprob}) for $r = 1, \dots, N$ now follows from (\ref{G}) and (\ref{upper}).  Since $N$ can be taken to be arbitrarily large, the result holds for all positive integers $r$.
\end{proof}

\section{Description of Example \ref{newex}}\label{newexsec}

We specify a random sequence $P_1 \geq P_2 \geq \dots$ such that $\sum_{j=1}^{\infty} P_j = 1$ a.s. in the following way:
\begin{enumerate}
\item Begin with any deterministic sequence $q_1 \geq q_2 \geq \dots$ such that $\sum_{j=1}^{\infty} q_j < 1/2$ and such that if $g(x) = \max\{j: q_j \geq x\}$, then $\lim_{x \rightarrow 0} x^{\alpha} g(x) = 1$.

\item Given a positive integer $n_1$, we can define, for $k \geq 2$, the integer $n_k = 2^{2^{2^{n_{k-1}}}}$.  Choose $n_1$ large enough that $\sum_{k=1}^{\infty} n_k^{\alpha - 1} < 1/2$.

\item Define a sequence of independent random variables $(R_k)_{k=1}^{\infty}$ such that $R_k$ has the uniform distribution on $\{1, 2, \dots, n_k\}$ for all $k$.  Then for all $k \in \N$, add the number $1/(n_k 2^{2^{R_k}})$ to the sequence $\lfloor 2^{2^{R_k}} n_k^{\alpha} \rfloor$ times.

\item Add the number $$1 - \sum_{j=1}^{\infty} q_j - \sum_{k=1}^{\infty} \frac{1}{n_k 2^{2^{R_k}}} \lfloor 2^{2^{R_k}} n_k^{\alpha} \rfloor$$ to the sequence to make the numbers sum to one.

\item Order the numbers and relabel them $P_1 \geq P_2 \geq \dots$.
\end{enumerate}
Using the method described in the introduction, define an exchangeable random partition $\Pi$ whose asymptotic block frequencies are almost surely given by this sequence $(P_j)_{j=1}^{\infty}$.  The next two lemmas show that $\Pi$ satisfies the conditions of Example \ref{newex}.

\begin{Lemma}\label{Gproblem}
For the sequence $(P_j)_{j=1}^{\infty}$ defined above, if we define $G(x) = \max\{j: P_j \geq x\}$, then $$\lim_{x \rightarrow 0} x^{\alpha} G(x) = 1 \hspace{.1in}\textup{in probability}.$$
\end{Lemma}

\begin{proof}
Let $G'(x)$ denote the number of terms that were added to the sequence in step 3 of the above construction that are greater than or equal to $x$.  Because $\lim_{x \rightarrow 0} x^{\alpha} g(x) = 1$ by step 1 of the construction, it suffices to show that
\begin{equation}\label{Gprimeprob}
\lim_{x \rightarrow 0} x^{\alpha} G'(x) = 0 \hspace{.1in}\textup{in probability}.
\end{equation}
Let $\epsilon > 0$.  Suppose $1/n_{k+1} \leq x \leq 1/n_k$.  Because $R_j \leq n_j$, there can be at most $2^{2^{n_j}} n_j^{\alpha}$ terms in the sequence that equal $1/(n_j 2^{2^{R_j}})$ for $j = 1, \dots, k-1$.  Therefore,
\begin{equation}\label{Gprime}
G'(x) \leq \sum_{j=1}^{k-1} 2^{2^{n_j}} n_j^{\alpha} + 2^{2^{R_k}} n_k^{\alpha} {\bf 1}_{\{1/(n_k 2^{2^{R_k}}) \geq x\}}.
\end{equation}
By the choice of $n_k$, we have $$\sum_{j=1}^{k-1} 2^{2^{n_j}} n_j^{\alpha} \leq \frac{\epsilon}{2} n_k^{\alpha} \leq \frac{\epsilon}{2} x^{-\alpha}$$ for sufficiently large $k$.  The second term on the right-hand side of (\ref{Gprime}) will be at most $(\epsilon/2) x^{-\alpha}$ unless we have both $1/(n_k 2^{2^{R_k}}) \geq x$ and $2^{2^{R_k}} n_k^{\alpha} \geq (\epsilon/2) x^{-\alpha}$ or, equivalently, unless
$$\log_2 \log_2 \bigg( \frac{\epsilon}{2 x^{\alpha} n_k^{\alpha}} \bigg) \leq R_k \leq \log_2 \log_2 \bigg( \frac{1}{x n_k} \bigg).$$  Because $R_k$ has a uniform distribution on $\{1, \dots, n_k\}$, the probability that $R_k$ falls in this interval is at most 
\begin{equation}\label{Rkprob}
\frac{1}{n_k} \bigg(1 + \log_2 \log_2 \bigg( \frac{1}{x n_k} \bigg) - \log_2 \log_2 \bigg( \frac{\epsilon}{2 x^{\alpha} n_k^{\alpha}} \bigg) \bigg).
\end{equation}
Note that for all real numbers $z > 1$, we have $$\log_2 \log_2 z - \log_2 \log_2 z^{\alpha} = \log_2 \bigg( \frac{\log_2 z}{\log_2 z^{\alpha}} \bigg) = \log_2 \bigg( \frac{1}{\alpha} \bigg).$$  By applying this result when $z = 1/(x n_k)$, we see that the probability in (\ref{Rkprob}) tends to zero as $k \rightarrow \infty$.  It follows that $\lim_{x \rightarrow \infty} P(G'(x) > \epsilon x^{-\alpha}) = 0$ for all $\epsilon > 0$, and (\ref{Gprimeprob}) follows.
\end{proof}

\begin{Lemma}
For the random partition $\Pi$ defined above, if $\Pi_n$ denotes the restriction of $\Pi$ to $\{1, \dots, n\}$ and $K_n$ denotes the number of blocks of $\Pi_n$, then there exists a constant $C > 0$ such that $$\lim_{k \rightarrow \infty} P\big(n_k^{-\alpha} K_{n_k} \geq \Gamma(1 - \alpha) + C \big) = 1.$$ 
\end{Lemma}

\begin{proof}
We use Poissonization.  Let $(N(t), t \geq 0)$ be a rate one Poisson process, and let $\Phi(t) = E[K_{N(t)}|(P_j)_{j=1}^{\infty}]$.  By (\ref{B}), it suffices to show that there is a $C > 0$ such that
\begin{equation} \label{phists}
\liminf_{k \rightarrow \infty} n_k^{-\alpha} \Phi(n_k) \geq \Gamma(1 - \alpha) + C \hspace{.1in}\textup{a.s.}
\end{equation}
For all $k \in \N$, designate $\lfloor 2^{2^{R_k}} n_k^{\alpha} \rfloor$ blocks of $\Pi$ with asymptotic frequency $1/(n_k 2^{2^{R_k}})$ as marked blocks, while the other blocks of $\Pi$ will be unmarked.  If there are more than $\lfloor 2^{2^{R_k}} n_k^{\alpha} \rfloor$ blocks with asymptotic frequency $1/(n_k 2^{2^{R_k}})$ because $q_j = 1/(n_k 2^{2^{R_k}})$ for some $j$, then choose at random the blocks to mark.  Note that the marked blocks correspond to the terms $P_k$ that were added in step 3 of the above construction.  The unmarked blocks all have asymptotic frequency $q_j$ for some $j$, except for the block added in step 4 of the construction.  Let $\Phi'(t)$ be the expected number of marked blocks of $\Pi_{N(t)}$ conditional on $(P_j)_{j=1}^{\infty}$, and let $\Phi''(t)$ be the expected number of unmarked blocks of $\Pi_{N(t)}$ conditional on $(P_j)_{j=1}^{\infty}$.  Note that $\Phi(t) = \Phi'(t) + \Phi''(t)$.  By Proposition \ref{karlin1} and (\ref{B}), we have
\begin{equation}\label{phi''}
\lim_{k \rightarrow \infty} n_k^{-\alpha} \Phi''(n_k) = \Gamma(1 - \alpha) \hspace{.1in}\textup{a.s.}
\end{equation}
The number of integers in the set $\{1, \dots, N(n_k)\}$ that are in a block of $\Pi$ with asymptotic frequency $1/(n_k 2^{2^r})$ has a Poisson distribution with mean $2^{-2^r}$.  Therefore, on the event $\{R_k = r\}$, we have $$\Phi'(n_k) \geq \lfloor 2^{2^r} n_k^{\alpha} \rfloor (1 - e^{-2^{-2^r}}).$$  Since $x^{-1}(1 - e^{-x})$ is bounded away from zero for all $x \leq 1/4$, it follows that there is a constant $C > 0$ such that $n_k^{-\alpha} \Phi'(n_k) \geq C$ a.s. for all $k$.  This fact, combined with (\ref{phi''}), implies (\ref{phists}).
\end{proof}

\section{Description of Example \ref{boszex}}\label{th2sec}

We begin by specifying a deterministic sequence of numbers $p_1 \geq p_2 \geq \dots$ such that $\sum_{j=1}^{\infty} p_j = 1$ as follows:
\begin{enumerate}
\item Begin with any sequence $q_1 \geq q_2 \geq \dots$ such that if $g(x) = \max\{j: q_j \geq x\}$, then
\begin{equation}\label{qlim}
\lim_{x \rightarrow 0} x (\log x)^2 g(x) = 1.
\end{equation}
It is not difficult to see that such sequences exist.  One arises, for example, in \cite{bago07}.

\item Choose any integer $j$ such that $$\sum_{k=j+1}^{\infty} q_k < 1 - \sum_{n=2}^{\infty} n^{-9/2}.$$  Then remove the terms $q_1, \dots q_j$ from the sequence.

\item For all $n \geq 2$, add the number $e^{-n^3}$ to the list $\lfloor n^{-9/2} e^{n^3} \rfloor$ times.

\item Add the number $$1 - \sum_{k=j+1}^{\infty} q_k - \sum_{n=2}^{\infty} e^{-n^3} \lfloor n^{-9/2} e^{n^3} \rfloor$$ to the sequence to make the numbers in the new sequence sum to one.

\item Order the numbers and relabel them $p_1 \geq p_2 \geq \dots$.
\end{enumerate}
Using the method described in the introduction, define an exchangeable random partition $\Pi$ whose asymptotic block frequencies are almost surely given by this sequence $(p_j)_{j=1}^{\infty}$.  The next two lemmas establish that $\Pi$ satisfies the conditions of Example \ref{boszex}.

\begin{Lemma}
For the random partition $\Pi$ defined above, if $\Pi_n$ denotes the restriction of $\Pi$ to $\{1, \dots, n\}$ and $K_n$ denotes the number of blocks of $\Pi_n$, then $$\lim_{n \rightarrow \infty} \frac{(\log n) K_n}{n} = 1 \hspace{.1in}\textup{a.s.}$$
\end{Lemma}

\begin{proof}
We again use Poissonization.  Let $(N(t), t \geq 0)$ be a rate one Poisson process, and let $\Phi(t) = E[K_{N(t)}]$.  By (\ref{B}), it suffices to show that
\begin{equation}\label{phi}
\lim_{t \rightarrow \infty} \frac{(\log t) \Phi(t)}{t} = 1.
\end{equation}

For all $n \geq 2$, designate $\lfloor n^{-9/2} e^{n^3} \rfloor$ blocks of $\Pi$ with asymptotic frequency $e^{-n^3}$ as marked blocks, while the others are unmarked blocks.  If there are more than $\lfloor n^{-9/2} e^{n^3} \rfloor$ blocks with asymptotic frequency $e^{-n^3}$ because $q_k = e^{-n^3}$ for some $k$, then choose at random the blocks to mark.  Note that the marked blocks correspond to the terms $p_k$ that were added in step 3 of the construction above.  The unmarked blocks all have asymptotic frequency $q_k$ for some $k > j$, except for the one unmarked block that is added in step 4 of the construction.  Let $\Phi'(t)$ be the expected number of marked blocks of $\Pi_{N(t)}$, and let $\Phi''(t)$ be the expected number of unmarked blocks of $\Pi_{N(t)}$.  Note that $\Phi(t) = \Phi'(t) + \Phi''(t)$.  In view of (\ref{qlim}), we can apply Proposition \ref{karlin2} with $\ell(t) = (\log t)^{-2}$ for $t > 1$ in combination with (\ref{B}) to get
\begin{equation}\label{phi2}
\lim_{t \rightarrow \infty} \frac{(\log t) \Phi''(t)}{t} = 1.
\end{equation}
That $q_1, \dots, q_j$ were deleted and one unmarked block was added does not affect this conclusion.

Now, choose $t$ such that $e^{(n-1)^3} < t \leq e^{n^3}$.  The number of marked blocks of $\Pi$ with asymptotic frequency at least $e^{-(n-1)^3}$ is $$\sum_{k=1}^{n-1} \lfloor k^{-9/2} e^{k^3} \rfloor \leq C_1 n^{-9/2} e^{(n-1)^3} \leq C_1 n^{-9/2} t,$$ where $C_1$ is a positive constant that does not depend on $n$.  This bound holds because the sum is dominated by the largest term.  If a block of $\Pi$ has asymptotic frequency $q$, then the probability that at least one of the first $N(t)$ integers is in the block is $1 - e^{-qt} \leq qt$.  Therefore, the expected number of marked blocks of $\Pi_{N(t)}$ with asymptotic frequency $e^{-n^3}$ or smaller is at most $$\sum_{k=n}^{\infty} (e^{-k^3} t) \cdot  k^{-9/2} e^{k^3} = t \sum_{k=n}^{\infty} k^{-9/2} \leq C_2 n^{-7/2} t,$$ where $C_2$ is another positive constant that does not depend on $n$.  Therefore, $\Phi'(t) \leq C_1 n^{-9/2} t + C_2 n^{-7/2} t.$
Since $\log t \leq n^3$, it follows that
\begin{equation}\label{phi1}
\lim_{t \rightarrow \infty} \frac{(\log t) \Phi'(t)}{t} = 0.
\end{equation}
Now (\ref{phi}) follows from (\ref{phi1}) and (\ref{phi2}).
\end{proof}

\begin{Lemma}
For the random partition $\Pi$ defined above, if $\Pi_n$ denotes the restriction of $\Pi$ to $\{1, \dots, n\}$ and $K_{n,r}$ denotes the number of blocks of $\Pi_n$ of size $r$, then for $r \geq 2$, the quantity $n^{-1} (\log n)^2 K_{n, r}$ does not converge to $1/[r(r-1)]$ in probability as $n \rightarrow \infty$.
\end{Lemma}

\begin{proof}
We consider the sequence $(K_{\lfloor m_n \rfloor, r})_{n=1}^{\infty}$, where $m_n = e^{n^3}$ for all $n$.  Let $(N(t), t \geq 0)$ be a rate one Poisson process.  There are at least $\lfloor n^{-9/2} e^{n^3} \rfloor$ blocks of $\Pi$ with asymptotic frequency $e^{-n^3}$.  Order these blocks at random, and then let $A_{i,n}$ be the event that the $i$th of these blocks contains exactly $r$ of the integers $1, \dots, N(m_n)$.  Because the number of the integers $\{1, \dots, N(n_m)\}$ in one of these blocks has a Poisson distribution with mean $1$, we have $P(A_{i,n}) = e^{-1}/r!$ for all $i$ and $n$.  Also, for any $n$, the events $A_{i,n}$ for $1 \leq i \leq \lfloor n^{-9/2} e^{n^3} \rfloor$ are independent.  It follows that for all $n$, the random variable $K_{N(m_n), r}$ stochastically dominates a Binomial$(\lfloor n^{-9/2} e^{n^3} \rfloor, e^{-1}/r!)$ random variable.  It now follows from standard large deviations estimates that
\begin{equation}\label{Knlim}
\lim_{n \rightarrow \infty} P\bigg(K_{N(m_n), r} \geq \frac{2e^{-1}}{3r!} n^{-9/2} e^{n^3} \bigg) = 1.
\end{equation}
Because $N(m_n)$ has the Poisson distribution with mean $m_n$, we have $\mbox{Var}(N(m_n)) = m_n$ and therefore $E[|N(m_n) - m_n|] \leq m_n^{1/2}$.  Since $|K_{N(m_n), r} - K_{\lfloor m_n \rfloor, r}| \leq |N(m_n) - \lfloor m_n \rfloor|$, it follows that $E[|K_{N(m_n), r} - K_{\lfloor m_n \rfloor, r}|] \leq e^{n^3/2} + 1$.  Combining this result with Markov's inequality gives
\begin{equation}\label{Knlim2}
\lim_{n \rightarrow \infty} P\bigg(|K_{N(m_n), r} - K_{\lfloor m_n \rfloor, r}| > \frac{e^{-1}}{3r!} n^{-9/2} e^{n^3} \bigg) = 0.
\end{equation}
Combining (\ref{Knlim}) and (\ref{Knlim2}) gives $$\lim_{n \rightarrow \infty} P\bigg(K_{\lfloor m_n \rfloor, r} \geq \frac{e^{-1}}{3r!} n^{-9/2} e^{n^3} \bigg) = 1.$$  Since $m_n (\log m_n)^{-2} = n^{-6} e^{n^3}$, the result follows.
\end{proof}

\section{Proof of Theorem \ref{coalth}}\label{appsec}

We will assume that $(\Psi_n(t), t \geq 0)$ is obtained from Kingman's coalescent $(\Theta_n(t), t \geq 0)$ as in (\ref{timechange}).  For any partition $\pi$ of $\{1, \dots, n\}$, let $|\pi|$ denote the number of blocks of $\pi$.  For $1 \leq k \leq n$, let $T_k = \inf\{t: |\Theta_n(t)| = k\}$.

\begin{Lemma}\label{Kingtube}
For all $\epsilon > 0$, there exists a positive constant $C$ such that with probability at least $1 - \epsilon$, we have $$\bigg| T_k - \bigg( \frac{2}{k} - \frac{2}{n} \bigg) \bigg| \leq \frac{C}{n^{9/8}}$$ for all integers $k$ such that $n^{3/4} \leq k \leq n$.
\end{Lemma}

\begin{proof}
If $2 \leq k \leq n$, then $T_{k-1} - T_k$ has an exponential distribution with rate $\binom{k}{2}$.  Since $T_n = 0$, it follows that $$E[T_k] = \sum_{j = k+1}^n E[T_{j-1} - T_j] = \sum_{j=k+1}^n \frac{2}{j(j-1)} = \sum_{j=k+1}^n \bigg( \frac{2}{j-1} - \frac{2}{j} \bigg) = \frac{2}{k} - \frac{2}{n}.$$  For $1 \leq k \leq n$, let $Y_k = T_k - E[T_k]$.  Note that $Y_{k-1} - Y_k = T_{k-1} - T_k - 2/[k(k-1)]$, and these increments are independent.  Therefore, $$\mbox{Var}(Y_k) = \sum_{j = k+1}^n \mbox{Var}(Y_{j-1} - Y_j) = \sum_{j=k+1}^n \mbox{Var}(T_{j-1} - T_j) = \sum_{j = k+1}^n \frac{4}{j^2(j-1)^2} \leq \frac{C_1}{k^3}$$ for some positive constant $C_1$.  By Kolmogorov's Maximal Inequality,
$$P \bigg( \max_{k \leq j \leq n} |Y_j| > \frac{C}{k^{3/2}} \bigg) \leq \frac{k^3}{C^2} \cdot \frac{C_1}{k^3} = \frac{C_1}{C^2},$$ which is less than $\epsilon$ if we take $C$ sufficiently large.  The result follows by taking $k = \lceil n^{3/4} \rceil$, in which case $C/k^{3/2} \leq C/n^{9/8}$.
\end{proof}

For $1 \leq k \leq n$, let $U_k = \inf\{t: |\Psi_n(t)| = k\}$.  Define the function $g: [0, \infty) \rightarrow [0, \infty)$ by $g(t) = (1 - \alpha)^{-(1-\alpha)} t^{1-\alpha}$, where $\alpha = \gamma/(1 + \gamma)$.  It follows from (\ref{timechange}) that for all $t \geq 0$,
$$\Psi_n(g(t)) = \Theta_n \bigg( \frac{g(t)^{\gamma + 1}}{\gamma + 1} \bigg) = \Theta_n \bigg( \frac{t^{(1-\alpha)(\gamma + 1)}}{(1-\alpha)^{(1-\alpha)(\gamma + 1)} (\gamma + 1)} \bigg) = \Theta_n(t).$$  Therefore, $U_k = g(T_k)$ for all $k$.

Let $$L_n = \sum_{k=2}^n k (U_{k-1} - U_k).$$  Note that $L_n$ is the sum of the lengths of all branches in the coalescent tree because $U_{k-1} - U_k$ is the amount of time for which there are exactly $k$ lineages.  Let $m = \lceil n^{3/4} \rceil + 1$, and let $$L'_n = \sum_{k = m}^n k (U_{k-1} - U_k),$$ which is the total length of all branches in the coalescent tree when the tree is truncated at the point where the number of lineages reaches $\lceil n^{3/4} \rceil$.

\begin{Lemma}\label{Lprimelem}
We have $$\lim_{n \rightarrow \infty} \frac{L_n'}{n^{\alpha}} = \frac{2^{1-\alpha} (1 - \alpha)^{\alpha} \pi}{\sin(\pi \alpha)} \hspace{.1in}\textup{in probability}.$$
\end{Lemma}

\begin{proof}
Let $\epsilon > 0$.  By Lemma \ref{Kingtube}, there is a constant $C$ such that with probability $1 - \epsilon$, we have $$\frac{2}{k} - \frac{2}{n} - \frac{C}{n^{9/8}} \leq T_k \leq \frac{2}{k} - \frac{2}{n} + \frac{C}{n^{9/8}}$$
whenever $n^{3/4} \leq k \leq n$.  Note that $U_n = 0$ and $L_n'$ is an increasing function of $U_k$ for $m-1 \leq k \leq n-1$.  Therefore, with probability at least $1 - \epsilon$,
\begin{align}\label{Ln1}
L_n' &= \sum_{k=m}^n k(g(T_{k-1}) - g(T_k)) \nonumber \\
&\leq \sum_{k=m}^n k \bigg( g \bigg( \frac{2}{k-1} - \frac{2}{n} + \frac{C}{n^{9/8}} \bigg) - g \bigg( \frac{2}{k} - \frac{2}{n} + \frac{C}{n^{9/8}} \bigg) \bigg) \nonumber \\
&\leq n g \bigg( \frac{2}{n-1} - \frac{2}{n} + \frac{C}{n^{9/8}} \bigg) + \sum_{k=m}^{n-1} k \bigg( \frac{2}{k-1} - \frac{2}{k} \bigg) g' \bigg( \frac{2}{k} - \frac{2}{n} + \frac{C}{n^{9/8}} \bigg),
\end{align}
where the last equality uses that $g'(t)$ is a decreasing function of $t$ because $0 < \alpha < 1$.  The first term on the right-hand side of (\ref{Ln1}) is $O(n^{1 - 9(1-\alpha)/8})$ and therefore is $o(n^{\alpha})$.  Since $g'(t) = (1 - \alpha)^{\alpha} t^{-\alpha}$, the second term on the right-hand side of (\ref{Ln1}) is equal to
\begin{equation}\label{Ln2}
\sum_{k=m}^{n-1} \frac{2(1 - \alpha)^{\alpha}}{k-1} \bigg( \frac{2}{k} - \frac{2}{n} + \frac{C}{n^{9/8}} \bigg)^{-\alpha} \leq 2^{1-\alpha} (1 - \alpha)^{\alpha} \sum_{k=m}^{n-1} \frac{1}{k-1} \bigg( \frac{1}{k} - \frac{1}{n} \bigg)^{-\alpha}.
\end{equation}
For all $k$ such that $m \leq k \leq n-1$,
$$\frac{1}{k-1} \bigg( \frac{1}{k} - \frac{1}{n} \bigg)^{-\alpha} = \bigg( \frac{k+1}{k-1} \bigg) \frac{1}{k+1} \bigg( \frac{1}{k} - \frac{1}{n} \bigg)^{-\alpha} \leq \frac{m+1}{m-1} \int_k^{k+1} \frac{1}{x} \bigg( \frac{1}{x} - \frac{1}{n} \bigg)^{-\alpha} \: dx.$$  Therefore, the second term on the right-hand side of (\ref{Ln1}) is at most $$2^{1-\alpha} (1-\alpha)^{\alpha} \bigg( \frac{m-1}{m+1} \bigg) \int_0^n \frac{1}{x} \bigg( \frac{1}{x} - \frac{1}{n} \bigg)^{-\alpha} \: dx.$$
By making the substitution $y = x/n$, we get
\begin{align}\label{integral}
\int_0^n \frac{1}{x} \bigg( \frac{1}{x} - \frac{1}{n} \bigg)^{-\alpha} \: dx &= n^{\alpha} \int_0^1 \frac{1}{y} \bigg( \frac{1}{y} - 1 \bigg)^{-\alpha} \: dy = n^{\alpha} \int_0^1 y^{\alpha - 1} (1-y)^{-\alpha} \: dy \nonumber \\
&= n^{\alpha} \Gamma(\alpha) \Gamma(1 - \alpha) = \frac{\pi n^{\alpha}}{\sin(\pi \alpha)},
\end{align}
where the last step uses Euler's Reflection Formula (see, for example, p. 9 of \cite{aar}).  Therefore, there exists a sequence $(a_n)_{n=1}^{\infty}$ tending to zero such that with probability at least $1 - \epsilon$, 
\begin{equation}\label{Lnupper}
\frac{L_n'}{n^{\alpha}} \leq \frac{2^{1-\alpha} (1 - \alpha)^{\alpha} \pi}{\sin(\pi \alpha)} + a_n.
\end{equation}

Likewise, for the lower bound, let $M = \max\{k: 2/k - 2/n - C/n^{9/8} > 0\}$.  Then with probability at least $1 - \epsilon$, we have
\begin{align}
L_n' &= \sum_{k=m}^n k(g(T_{k-1}) - g(T_k)) \nonumber \\
&\geq \sum_{k=m}^M k \bigg( g \bigg( \frac{2}{k-1} - \frac{2}{n} - \frac{C}{n^{9/8}} \bigg) - g \bigg( \frac{2}{k} - \frac{2}{n} - \frac{C}{n^{9/8}} \bigg) \bigg) \nonumber \\
&\geq \sum_{k=m}^M k \bigg( \frac{2}{k-1} - \frac{2}{k} \bigg) g' \bigg( \frac{2}{k-1} - \frac{2}{n} - \frac{C}{n^{9/8}} \bigg) \nonumber \\
&\geq \sum_{k=m}^M \frac{2}{k-1} (1 - \alpha)^{\alpha} \bigg( \frac{2}{k-1} - \frac{2}{n} \bigg)^{-\alpha} \nonumber \\
&= 2^{1-\alpha} (1 - \alpha)^{\alpha} \sum_{m-1}^{M-1} \frac{1}{k} \bigg( \frac{1}{k} - \frac{1}{n} \bigg)^{-\alpha}. \nonumber
\end{align}
For all $k$ such that $m-1 \leq k \leq M-1$,
$$\frac{1}{k} \bigg( \frac{1}{k} - \frac{1}{n} \bigg)^{-\alpha} = \bigg( \frac{k-1}{k} \bigg) \frac{1}{k-1} \bigg( \frac{1}{k} - \frac{1}{n} \bigg)^{-\alpha} \geq \frac{m-2}{m-1} \int_{k-1}^k \frac{1}{x} \bigg( \frac{1}{x} - \frac{1}{n} \bigg)^{-\alpha} \: dx.$$
Since $m/n \rightarrow 0$ and $M/n \rightarrow 1$ as $n \rightarrow \infty$, it now follows from (\ref{integral}) that there is a sequence $(b_n)_{n=1}^{\infty}$ tending to zero such that with probability at least $1 - \epsilon$,
\begin{equation}\label{Lnlower}
\frac{L_n'}{n^{\alpha}} \geq \frac{2^{1-\alpha} (1 - \alpha)^{\alpha} \pi}{\sin(\pi \alpha)} - b_n.
\end{equation}
The result now follows from (\ref{Lnupper}) and (\ref{Lnlower}).
\end{proof}

\begin{Lemma}\label{LnLem}
We have $$\lim_{n \rightarrow \infty} \frac{L_n}{n^{\alpha}} = \frac{2^{1-\alpha} (1 - \alpha)^{\alpha} \pi}{\sin(\pi \alpha)} \hspace{.1in}\textup{in probability}.$$
\end{Lemma}

\begin{proof}
By Lemma \ref{Lprimelem}, it suffices to show that
\begin{equation}\label{Ldiff}
\lim_{n \rightarrow \infty} \frac{L_n - L_n'}{n^{\alpha}} = 0 \hspace{.1in}\textup{in probability}.
\end{equation}
We have $$L_n - L_n' = \sum_{k=2}^{m-1} k (g(T_{k-1}) - g(T_k)) \leq \sum_{k=2}^{m-1} k g'(T_{m-1})(T_{k-1} - T_k).$$  Let $A$ be the event that $T_{m-1} \geq 2/(m-1) - 2/n - C/n^{9/8}$, which has probability at least $1 - \epsilon$ by Lemma \ref{Kingtube}.  There is a positive constant $C_2$ such that $g'(T_{m-1}) \leq C_2 n^{3 \alpha/4}$ on $A$ for all $n$.  Therefore,
$$E[L_n - L_n'|A] \leq C_2 n^{3 \alpha/4} \sum_{k=2}^{m-1} k E[T_{k-1} - T_k] \leq C_2 n^{3 \alpha/4} \sum_{k=2}^{m-1} \frac{2}{k-1} \leq C_2 n^{3 \alpha/4} (1 + \log n).$$  Thus, by Markov's Inequality,
$$P(L_n - L_n' > \epsilon n^{\alpha}) \leq P(A^c) + \frac{E[L_n - L_n'|A]}{\epsilon n^{\alpha}} \leq \epsilon + \frac{C_2}{\epsilon} n^{-\alpha/4} (1 + \log n),$$ which is less than $2 \epsilon$ for sufficiently large $n$.  The result follows. 
\end{proof}

Recall that $L_n$ is the sum of the lengths of all branches in the coalescent tree.  Also, recall that mutations occur along each branch of the coalescent tree at times of a Poisson process of rate $\theta$.  Therefore, if we denote by $S_n$ the number of mutations in the tree, then conditional on $L_n$, the distribution of $S_n$ is Poisson with mean $\theta L_n$.  Thus, Lemma \ref{LnLem} and Chebyshev's Inequality immediately yield the following result.

\begin{Cor}\label{SCor}
We have $$\lim_{n \rightarrow \infty} \frac{S_n}{n^{\alpha}} = \frac{\theta 2^{1-\alpha} (1 - \alpha)^{\alpha} \pi}{\sin(\pi \alpha)} \hspace{.1in}\textup{in probability}.$$
\end{Cor}

Theorem \ref{coalth} now follows from Corollary \ref{SCor} and the next lemma.

\begin{Lemma}
$$\lim_{n \rightarrow \infty} \frac{S_n - K_n}{n^{\alpha}} = 0 \hspace{.1in}\textup{in probability}.$$
\end{Lemma}

\begin{proof}
Note that if the most recent mutation inherited by two sampled individuals is the same, then all of the mutations inherited by these individuals must be the same.  This is because when we follow the two lineages backwards in time, they must coalesce before any mutations are observed.  Therefore, each block of the allelic partition $\Pi_n$ can be associated with a mutation that is the most recent mutation inherited by the individuals in that block, with the possible exception of one block corresponding to individuals with no mutations.  It follows that $K_n \leq S_n + 1$.

To get a bound in the other direction, note that the only mutations that are not associated with a block of the allelic partition as above are the mutations that are not the most recent mutation inherited by any individual.  We denote the number of such mutations by $B_n$.  Then $K_n \geq S_n - B_n$, so it suffices to show that $B_n/n^{\alpha}$ converges in probability to zero as $n \rightarrow \infty$.

Let $R_n$ denote the number of mutations that occur when the number of lineages is $\lceil n^{3/4} \rceil$ or fewer.  Enumerate the remaining mutations in decreasing order of time, so that the first mutation is the most recent one, the second mutation is the second most recent, and so on.  Let $R_{k,n}$ denote the number of mutations along the branch of the coalescent tree that we get by starting at the $k$th mutation and following this lineage back until time $$g \bigg( \frac{2}{m-1} - \frac{2}{n} + \frac{C}{n^{9/8}} \bigg),$$ where $C$ is the constant from Lemma \ref{Kingtube}.  Choose $C_3 > \theta 2^{1-\alpha} (1 - \alpha)^{\alpha} \pi /(\sin(\pi \alpha))$.  On the event that $T_{m-1} \leq 2/(m-1) - 2/n + C/n^{9/8}$, which has probability at least $1 - \epsilon$ by Lemma \ref{Kingtube}, and on the event that $S_n \leq C_3 n^{\alpha}$, which has probability tending to one as $n \rightarrow \infty$ by Corollary \ref{SCor}, we have
\begin{equation}\label{BnRn}
B_n \leq R_n + \sum_{k=1}^{\lfloor C_3 n^{\alpha} \rfloor} R_{k,n}.
\end{equation}

Conditional on $L_n$ and $L_n'$, the distribution of $R_n$ is Poisson with mean $\theta (L_n - L_n')$.  Therefore, by (\ref{Ldiff}), $R_n/n^{\alpha}$ converges to zero in probability as $n \rightarrow \infty$.  Because mutations occur along each lineage at rate $\theta$, we have for all $k \leq \lfloor C_3 n^{\alpha} \rfloor$, $$E[R_{k,n}] \leq \theta g \bigg( \frac{2}{m-1} - \frac{2}{n} + \frac{C}{n^{9/8}} \bigg) \leq C_4 n^{-3(1-\alpha)/4}$$ for some positive constant $C_4$.  By summing over $k$ and then applying Markov's Inequality, we get that $$\lim_{n \rightarrow \infty} \frac{1}{n^{\alpha}} \sum_{k=1}^{\lfloor C_3 n^{\alpha} \rfloor} R_{k,n} = 0 \hspace{.1in}\textup{in probability}.$$  Since (\ref{BnRn}) holds with probability at least $1 - 2 \epsilon$ for sufficiently large $n$, the result follows.
\end{proof}

\clearpage
\noindent {\bf \Large Acknowledgments}
\bigskip

\noindent The author thanks Alexander Gnedin and Christina Goldschmidt, who posed the problem of extending Propositions \ref{karlin1} and \ref{karlin2} at an Oberwolfach meeting in 2007.  He also thanks Vlada Limic and a referee for comments on a previous version.

\end{document}